\documentclass[a4paper,11pt,english,reqno]{amsart}

\usepackage{amsmath,amssymb,amsfonts,dsfont,amsthm,upgreek,bm,mathtools}    
\usepackage[utf8]{inputenc}
\usepackage[english]{babel}
\usepackage{etoolbox}
\usepackage[T1]{fontenc} 
\usepackage{graphicx,xcolor}
\usepackage{pdfpages}
\usepackage{cite}
\usepackage{float}
\usepackage[font=small,labelfont=bf]{caption}
\usepackage[a4paper, margin=2.4cm]{geometry}
\usepackage{tikz,pgfplots}
\usepackage{tkz-fct}
\usepgfplotslibrary{polar}
\pgfplotsset{compat=1.18}

\usepackage{enumerate}
\usepackage{enumitem}
\newcommand{\stkout}[1]{\ifmmode\text{\sout{\ensuremath{#1}}}\else\sout{#1}\fi}

\newcommand{\prob}{\mathbf{P}}
\newcommand{\erw}{\mathds{E}}
\newcommand{\e}{\mathrm{e}}
\newcommand{\cH}{\mathcal{H}}

\newcommand{\WF}{\mathrm{WF}}

\newcommand{\spec}{\mathrm{Spec}}
\newcommand{\Ima}{\mathrm{Im\,}}
\newcommand{\Rea}{\mathrm{Re\,}}

\newcommand{\mO}{\mathcal{O}}
\newcommand{\C}{\mathds{C}}
\newcommand{\R}{\mathds{R}}
\newcommand{\T}{\mathds{T}}
\newcommand{\N}{\mathds{N}}
\newcommand{\Z}{\mathds{Z}}

\newcommand{\wt}{\widetilde}

\newcommand{\Vol}{\mathrm{Vol}}

\newcommand{\cS}{\mathcal{S}}

\renewcommand{\geq}{\geqslant}

\renewcommand{\leq}{\leqslant}

\newtheorem{thm}{Theorem}

\newtheorem{defn}[thm]{Definition}
\newtheorem{rem}[thm]{Remark}
\newtheorem{ex}[thm]{Example}
\numberwithin{equation}{section}
\numberwithin{thm}{section}

\usepackage{hyperref}
\hypersetup{pdfborder=0 0 0, 
	    colorlinks=true,
	    citecolor=blue,
	    linkcolor=blue,
	    urlcolor=blue,
	    pdfauthor={Martin Vogel}
	   }
\begin{document}
\title[Disordered non-selfadjoint operators ]
{Pseudospectra and eigenvalue asymptotics for disordered non-selfadjoint operators 
in the semiclassical limit}
\author{Martin Vogel}
\address[Martin Vogel]{Institut de Recherche Math{\'e}matique Avanc{\'e}e - UMR 7501, 
Universit{\'e} de Strasbourg et CNRS, 7 rue René-Descartes, 67084 Strasbourg Cedex, France.}
\email{vogel@math.unistra.fr}
%
%
\maketitle
\begin{abstract} 
The purpose of this note is to review some recent results concerning the pseudospectra 
and the eigenvalues asymptotics of non-selfadjoint semiclassical pseudo-differential 
operators subject to small random perturbations. 
\end{abstract}
\section{Introduction}
Spectral theory of non-selfadjoint operators acting on a Hilbert 
space is an old and highly developed subject. Non-selfadjoint operators 
appear naturally in large variety of modern problems. For instance in 
quantum mechanics, the study of scattering systems naturally leads to 
the notion of quantum resonances. These can be described as the complex 
valued poles of the meromorphic continuation of the scattering matrix 
or of the cut-off resolvent of the Hamiltonian to the non-physical sheet of the 
complex plane. Alternatively, by means of a complex deformation of the 
initial Hamiltonian, these resonances can be described as the genuine 
complex valued eigenvalues of a non-selfadjoint operator 
\cite{AgCo71,BaCo71,SjZw91}. We refer the reader to \cite{DyZw19} for an 
in-depth discussion of the mathematics of scattering poles. 
Another subject in quantum mechanics is the study of a \emph{small} system 
which is linked to a \emph{larger} environment. The effective dynamics of 
the small systems are governed by a non-selfadjoint operator: the Lindbladian 
\cite{Lind76}.
\par 
A major difficulty in the spectral analysis of non-selfadjoint operators 
is the possible strong \emph{spectral instability} of their spectrum with respect 
to small perturbations. This phenomenon, sometimes called \emph{pseudospectral 
effect}, was initially considered a drawback as it can be at the origin 
of immense numerical errors, see \cite{TrEm05} and the references therein. 
However, a recent line of research has also shown that it can be at the 
pseudospectral effect can lead to new insights into the spectral distribution 
of non-selfadjoint operators subject to small generic perturbations. 
\section{Spectral instability of non-selfadjoint operators}
We begin by recalling the definition of the \emph{pseudospectrum} 
of a linear operator, an important notion which quantifies its spectral 
instability. This notion seems to have originated in the second half 
of the 20th century in various contexts, see \cite{Tr97} for a 
historic overview. It quickly became an important notion in numerical 
analysis as it allows to quantify how much eigenvalues can \emph{spread out}  
under the influence of small perturbations, see \cite{Tr92,Tr97} 
and the book \cite{TrEm05}. We follow here the latter reference. 
\\
\par
Let $\cH$ be a complex Hilbert space (assumed separable for simplicity) 
with norm $\|\cdot\|$ and scalar product $(\cdot|\cdot)$. Let $P:\cH \to \cH$ 
be a closed densely defined linear operator, with resolvent 
set $\rho(P)$ and spectrum $\spec(P)=\C\backslash\rho(P)$. 
\begin{defn}\label{def:PseudoSpec}
For any $\varepsilon>0$, we define the $\varepsilon$-pseudospectrum of $P$ by 
\begin{equation}\label{C1:eq1}
   \spec_{\varepsilon}(P) := \spec(P)\cup 
   \{z\in\rho(P); \|(P-z)^{-1}\| > \varepsilon^{-1}\}.
\end{equation}
\end{defn}
We remark that some authors define the $\varepsilon$-pseudospectrum 
with a $\geq$ rather than a $>$. We, however, follow here \cite{TrEm05}. 
Note that with this choice of non-strict inequality the $\spec_{\varepsilon}(P)$ 
is an open set in $\C$. 
\par 
For $P$ selfadjoint (or even normal), the spectral theorem implies that 
\begin{equation}\label{C1:eq1.1}
	\spec_{\varepsilon}(P) \subset \spec(P)\cup D(0,\varepsilon).
 \end{equation}
For $P$ non-selfadjoint, the pseudospectrum of $P$ can be much 
larger, as illustrated by the following example. 
\begin{ex}\label{C2:exJordan}
	For $N\gg 1$ consider the Jordan block matrix 
	\begin{equation}\label{ex:JordanBlock}
		P_N = \begin{pmatrix}
		 0& 1& 0 & \dots & 0 \\
		 0 & 0 & 1 & \ddots& \vdots\\
		   \vdots & \ddots & \ddots & \ddots& 0\\
		   \vdots&\dots&\ddots &\ddots&1 \\
		   0&\dots&\dots&\dots&0 
	   \end{pmatrix}: \C^N \to \C^N.
	\end{equation}
The spectrum of $P_N$ is given by $\{0\}$. Consider the vector 
$e_+=(1,z,\dots,z^{N-1})$, $|z|\leq r <1$. Then,
\begin{equation*}
	\|(P_N-z)e_+\|=|z|^N = \mO\!\left(\e^{-N|\log r|}\right) \|e_+\|.
\end{equation*} 
So, Theorem \ref{Thm:PSpecC3} below shows that for any $\varepsilon>0$ and 
any $r\in ]0,1[$ we have that for $N>1$ sufficiently large
\begin{equation*}
	D(0,r)\subset \spec_\varepsilon(P_N).
\end{equation*}
\end{ex}
An immediate consequence of \eqref{C1:eq1} is the property that pseudospectra 
are nested. More precisely,
\begin{equation}\label{C1:eq1.0}
   \spec_{\varepsilon_2}(P) \subset \spec_{\varepsilon_1}(P), \quad 
   \varepsilon_1 > \varepsilon_2>0.
\end{equation}
The set \eqref{C1:eq1} describes a region of spectral instability of the 
operator $P$, since any point in the 
$\varepsilon$-pseudospectrum of $P$ lies in the spectrum of a certain 
$\varepsilon$-perturbation of $P$ \cite{TrEm05}. 
\begin{thm}
 Let $\varepsilon>0$. Then 
 \begin{equation}\label{C1:eq2}
   \spec_{\varepsilon}(P)= \bigcup_{\substack{Q\in\mathcal{B}(\cH,\cH) \\
              \lVert Q\rVert < 1}}
          \spec(P+\varepsilon Q).
\end{equation}
\end{thm}
\begin{proof}
 See \cite[p. 31]{TrEm05}.
\end{proof}
A third, equivalent definition of the $\varepsilon$-pseudospectrum of 
$P$ is via the existence of approximate solutions to the eigenvalue 
problem $(P-z)u=0$. 
\begin{thm}\label{Thm:PSpecC3}
 Let $\varepsilon>0$ and $z\in\C$. Then the following statements are equivalent: 
 \begin{enumerate}
   \item $z \in \spec_{\varepsilon}(P)$;
   \item $z \in \spec(P)$ or there exists a $u_z\in \mathcal{D}(P)$ such 
   that $\|(P-z)u_z\| < \varepsilon \|u_z\|$, where $\mathcal{D}(P)$ denotes the 
   domain of $P$.
 \end{enumerate}
\end{thm}
\begin{proof}
 See \cite[p. 31]{TrEm05}.
\end{proof}
Such a state $u_z$ is called an $\varepsilon$-quasimode, or simply a 
\textit{quasimode} of $P-z$. 
\section[Spectral instability of semiclassical $\Psi$dos]
{Spectral instability of semiclassical pseudo-differential 
operators}\label{sec:SpecInstaSemPseudo}
Although the notion of $\varepsilon$-pseudospectrum defined in 
Definition \ref{def:PseudoSpec} is valid in the setting of semiclassical 
pseudo-differential operators, we present here a somewhat different, yet 
still related notion, which is more adapted to semiclassical setting. 
Here ``semiclassical'' means that our operators depend on a parameter 
$h\in ]0,1]$ (often referred to as
``Planck's parameter''), and that we will be interested in the asymptotic ({\it semiclassical
}) regime $h\searrow 0$. This small parameter will provide us with a natural threshold to
define the pseudospectrum, and thereby to measure the spectral
instability. 
The following discussion is based on the works by Davies \cite{Da99} and 
Dencker, Sj\"ostrand and Zworski \cite{NSjZw04}. 
\\
\par
Let $d\geq1$ and $h\in]0,1]$. An \emph{order function} 
$m\in\mathcal{C}^{\infty}(\R^{2d}; [1,\infty[ )$, is a 
function satisfying the following growth condition:
 \begin{equation}\label{C1:eq3}
 \exists C_0\geq 1,~ \exists N_0 >0 : \quad  
 m(\rho) \leq C_0 \langle \rho-\mu\rangle^{N_0} m(\mu),  
 \quad \forall \rho,\mu \in \R^{2d},
 \end{equation}
where $\langle \rho-\mu\rangle:=\sqrt{1+|\rho-\mu|^2}$ denotes the 
``Japanese brackets''. We will also sometimes write $(x,\xi)=\rho\in \R^{2d}$, 
so that $\xi\in \R^d$. To such an order function $m$ we may associate a semiclassical 
symbol class \cite{DiSj99,Zw12}. We say that a smooth function 
$p\in \mathcal{C}^{\infty}(\R^{2d}_\rho,]0,1]_h)$ is in the symbol class $S(m)$ 
if for any multiindex $\alpha \in \N^{2d}$ there exists a constant $C_\alpha>0$ such 
that 
\begin{equation}\label{C1:eq4}
|\partial^{\alpha}_{\rho}p(\rho;h)|\leq C_{\alpha}m(\rho),
 \quad\forall \rho\in\R^{2d},\ \forall h\in]0,1].
\end{equation}
We refer the reader for further reading on semiclassical analysis 
to \cite{DiSj99,Zw12,Ma02}.
\par
Let the symbol $p\in S(m)$, $m\geq 1$, be a ``classical'' symbol, namely it 
satisfies an asymptotic expansion in the limit $h\to 0$:
\begin{equation}\label{C1:eq5}
   p(\rho;h) \sim p_0(\rho)+ hp_1(\rho) + \dots \quad \text{in } S(m),
\end{equation}
where each $p_j\in S(m)$ is independent of $h$. We assume that there 
exists a $z_0\in \C$ and a $C_0>0$ such that 
\begin{equation}\label{C1:eq5.2}
   |p_0(\rho)-z_0| \geq m(\rho)/C_0, \quad \rho \in T^*\R^{d}. 
\end{equation}
In this case we call $p_0$ the (semiclassical) principal symbol of $p$. 
We then define two subsets of $\C$ associated with $p_0$:
\begin{equation}\label{C1:eq5.1}
       \Sigma:= \Sigma(p_0):=\overline{p_0(T^*\R^{d})} , \qquad
       \Sigma_{\infty} := \{
       z\in\Sigma; ~\exists (\rho_j)_{j\geq 1} \text{ s.t. } |\rho_j|\to\infty, ~ p_0(\rho_j)\to z
       \}.
\end{equation}
The set $\Sigma$ is the \emph{classical spectrum}, and $\Sigma_\infty$ can be 
called the \emph{classical spectrum at infinity} of the $h$-Weyl quantization of 
$p$ defined by
\begin{equation}\label{C1:eq8}
   P_hu(x):=p^w(x,hD_x)u(x) = \frac{1}{(2\pi h )^d}\iint\e^{\frac{i}{h}(x-y)\cdot \xi} 
       p\!\left(\frac{x+y}{2},\xi;h\right) u(y)dyd\xi, \quad u\in\cS(\R^d),
\end{equation}
seen as an oscillatory integral in $\xi$. The operator $P_h$ 
maps $\cS\to \cS$, and by duality $\cS'\to \cS'$, continuously. 
\subsection{Semiclassical pseudospectrum} \label{HabCh1:sec:SCPS}
Similar to \cite{NSjZw04}, we define for a symbol $p\in S(m)$ as in 
\eqref{C1:eq5} the sets 
\begin{equation}\label{C1:eq5.3}
 \Lambda_\pm(p):=\left\{ p(\rho); ~\pm\frac{1}{2i}\{\overline{p}, p\}(\rho)<0
   \right\}\subset \Sigma \subset \C,
\end{equation}
where $\{\cdot,\cdot\}$ denotes the Poisson bracket. Note that the condition 
$\frac{1}{2i}\{\overline{p}, p\}\neq 0$ is the classical analogue of the 
$[P_h^*,P_h]\neq 0$. As in \cite{NSjZw04} we call the set 
\begin{equation}\label{C1:eq5.4}
 \Lambda(p):= \overline{\Lambda_-\cup \Lambda_+}
\end{equation}
the \emph{semiclassical pseudospectrum}. 
\begin{thm}[{\cite{NSjZw04}}]\label{thm:DZS}
 Suppose that $n\geq 2$, $C^\infty_b(T^*\R^d) \ni p\sim p_0+hp_1+\dots$, and $p_0^{-1}(z)$ is 
 compact for a dense set of values $z\in\C$. If $P_h=p^w(x,hD_x)$, then 
 \begin{equation*}
     \Lambda(p_0)\backslash \Sigma_\infty \subset \overline{\Lambda_+(p_0)}
 \end{equation*}
 and for every $z\in\Lambda_+(p_0)$ and every $\rho_0\in T^*\R^d$ with 
 \begin{equation*}
   p_0(\rho_0)=z, \quad \frac{1}{2i}\{\overline{p}_0, p_0\}(\rho_0)<0,
 \end{equation*}
 there exists $0\neq e_+\in L^2(\R^d)$ such that 
 \begin{equation}\label{C1:eq5.5}
   \|(P_h -z)e_+\| = \mO(h^\infty)\|e_+\|, 
   \quad \WF_h(e_+)\footnote{This is to say that the semiclassical 
   wavefront set of $e_+$ is given by $\rho_0$. In other words, 
   the state $e_+$ is concentrated in position and frequency near 
   the point $\rho_0$. See for instance \cite{Zw12} for a definition.}
    = \{\rho_0\}. 
 \end{equation}
 If, in addition, $p$ has a bounded holomorphic continuation to to 
 $\{\rho\in \C^{2d}, ~|\Ima \rho | \leq 1/C\}$, then \eqref{C1:eq5.5} 
 holds with the $h^\infty$ replaced by $\exp(-1/(Ch))$. 
 \par
 If $n=1$, then the same conclusion holds, provide that in addition 
 to the general assumptions, each component of $\C\backslash \Sigma_\infty$ 
 has a nonempty intersection with $\complement \Lambda(p)$. 
\end{thm}
This result can be extended to unbounded symbols $p\in S(T^*\R^{d},m)$, as 
in \eqref{C1:eq5}, and the corresponding operators $P_h$ with principal symbol 
$p_0$, by applying Theorem \ref{thm:DZS} to $\wt{P}_h=(P_h-z_0)^{-1}(P_h-z)$, 
with principal symbol $\wt{p}_0\in C^\infty_b(T^*\R^d)$ and $z_0$ as 
in \eqref{C1:eq5.2} and $z_0\neq z$. Indeed, note that $z\in \Sigma(p_0)$ 
if and only if $0\in \Sigma(\wt{p}_0)$, and that $\rho \in p_0^{-1}(z)$ 
with $\pm\{\Rea p_0 , \Ima p_0\}(\rho)<0$ is equivalent to 
$\rho \in \wt{p}_0^{-1}(0)$ with $\pm\{\Rea \wt{p}_0 , \Ima \wt{p}_0\}(\rho)<0$. 
Furthermore, a quasimode $u$ as in Theorem \ref{thm:DZS} for $\wt{P}_h$ 
then provides, after a possible truncation, a quasimode for $P_h-z$ in the 
same sense. 
\\
\par 
By replacing $P_h$ with its formal adjoint $P_h^*$, 
and thus $p$ with $\overline{p}$, Theorem \ref{thm:DZS} yields that 
for every $z\in\Lambda_-(p)$ and every $\rho_0\in T^*\R^d$ with 
\begin{equation*}
 p_0(\rho_0)=z, \quad \frac{1}{2i}\{\overline{p}_0, p_0\}(\rho_0)>0,
\end{equation*}
there exists $0\neq e_-\in L^2(\R^d)$ such that 
\begin{equation*}
 \|(P_h -z)^*e_-\| = \mO(h^\infty)\|e_-\|, 
 \quad \WF_h(e_-) = \{\rho_0\}. 
\end{equation*}
The additional statements of Theorem \ref{thm:DZS} about symbols 
admitting a holomorphic extension to a complex neighborhood of $\R^{2d}$, 
and the case when $n=1$ hold as well. 
\begin{ex}\label{ex:DaviesOperator}
The guiding example to keep in mind is the case of the non-selfadjoint 
Harmonic oscillator 
\begin{equation*}
   P_h = (hD_x)^2 + i x^2
\end{equation*}
seen as an unbounded operators $L^2(\R)\to L^2(\R)$. The principal symbol 
of $P_h$ is given by $p(x,\xi)=\xi^2+ix^2\in S(T^*\R,m)$, with weight 
function $m(x,\xi) = 1+ \xi^2 + x^2$. We equip $P_h$ with the domain 
$H(m):=(P_h+1)^{-1}L^2(\R)$, where the operator on the right is the 
pseudo-differential inverse of $P_h+1$. This choice of domain makes $P_h$ 
a closed densely defined operator. Using, for instance, the method of complex scaling 
we see that the spectrum of $P_h$ is given by 
\begin{equation}\label{C1eq:DaviesOp1}
   \spec(P_h) = \{\e^{i\pi/4}(2n+1)h;n\in\N\}.
\end{equation}
Furthermore, $\Sigma$ is the closed 1st quadrant in the complex plane, 
whereas $\Sigma_\infty = \emptyset$. For $\rho=(x,\xi) \in T^*\R$, we find that 
\begin{equation}\label{C1eq:DaviesOp1.1}
   \frac{1}{2i}\{\overline{p}, p\}(x,\xi) = 2\xi\cdot x. 
\end{equation}
Thus, for every $z\in \mathring{\Sigma}$ there exist points 
\begin{equation*}
   \rho_{+}^j(z)= (-1)^j(-\sqrt{|\Rea z|},\sqrt{|\Ima z|}), \quad
   \rho_{-}^j(z)= (-1)^j(-\sqrt{|\Rea z|},-\sqrt{|\Ima z||}), \quad j=1,2,
\end{equation*}
such that 
\begin{equation*}
   \pm \frac{1}{2i}\{\overline{p}, p\}(\rho_{\pm}^j(z)) < 0, \quad j=1,2.
\end{equation*}
Using the WKB method, we can construct quasimodes of the form $e^j_+(x;h)=a_+^j(x;h)\e^{i\phi^j_+(x)/h}$ 
with $a_+^j(x;h)\in C^\infty_c(\R)$ admitting an asymptotic expansion $a_+^j(x;h) \sim a_{+,0}^j(x)
+ha_{+,1}^j(x)+\dots$ with $\WF_h(e_+^j)=\{ \rho_{+}^j(z)\}$ and 
\begin{equation}\label{C1:eq6}
   \|(P_h-z)e_+^j\| = \mO(\e^{1/Ch}),
\end{equation}
see \cite{Da99,Da99b} for an explicit computation, and \cite{NSjZw04} for a 
more general construction.
\end{ex}
In fact the works of Davies \cite{Da99,Da99b} provide an explicit WKB construction for a 
quasimode $u$ for one-dimensional non-selfadjoint Schr\"odinger operators $P_h-z=(hD_x)^2+V(x)-z$ 
on $L^2(\R)$ with $V\in C^\infty(\R)$ complex-valued and $z=V(a)+\eta^2$, for some $a\in\R$, $\eta>0$. 
Furthermore, one assumes that $\Ima V'(a)\neq 0$. These works were the starting point for the 
quasimode construction for non-selfadjoint (pseudo-)differential operators. Zworski \cite{Zw01} linked 
Davies' quasimode construction under the condition on the gradient of $\Ima V$ to a quasimode construction 
under a non-vanishing condition of the Poisson bracket $\frac{1}{2i}\{\overline{p}, p\}$. Furthermore, 
Zworski \cite{Zw01} established the link to the famous commutator condition of H\"ormander \cite{Hoe60,Hoe60b}. 
A full generalization of the quasimode construction under a non-vanishing condition of the poisson bracket, 
see Theorem \ref{thm:DZS} 
above, was then achieved by Dencker, Sj\"ostrand and Zworski \cite{NSjZw04}. Finally, Pravda-Starov 
\cite{Pr03,Pr06,Pr08} improved these results by modifying a quasimode construction by Moyer and H\"ormander, 
see \cite[Lemma 26.4.14]{Ho85}, for adjoints of operators that do not satisfy the Nirenberg-Trèves 
condition ($\Psi$) for local solvability. 
\par 
For a quasimode construction for non-selfadjoint boundary value problems we refer the 
reader to the work of Galkowski \cite{Ga14}. 
\\
\par
Notice, that \eqref{C1:eq5.5} (or \eqref{C1:eq6} in the example above) implies that 
if the resolvent $(P_h-z)^{-1}$ exists then it is larger than any power of $h$ when 
$h\to 0$, or even larger than $\e^{1/Ch}$ in the analytic case. We call each family 
$(e_{+}^j(z,h))$ an $h^{\infty}$-quasimode of $P_h-z$, or for short a quasimode of $P_h-z$.
\\
\par
From the quasimode equation \eqref{C1:eq5.5} it is easy to exhibit an operator $Q$ of 
unity norm and a parameter $\delta=\mO(h^{\infty})$, such that the perturbed operator 
$P_h+\delta Q$ has an eigenvalue at $z$. For instance, if we call the error 
$r_+=(P_h-z)e_+$, we may take the rank 1 operator $\delta Q=- r_+ \otimes (e_+)^*$. 
By Theorem \ref{thm:DZS} we see that the interior of the set $\Lambda(p)$, away from 
the set $\Sigma_\infty$, is a zone of strong spectral instability 
for $P_h$. For this reason we may refer to the semiclassical pseudospectrum $\Lambda(p)$ 
also as the ($h^\infty$-)pseudospectrum of $P_h$. Finally, we refer the reader also 
to the works of Pravda-Starov \cite{Pr03,Pr06,Pr08} for a refinement of the notion of 
semiclassical pseudospectrum. 
\subsection{Outside the semiclassical pseudospectrum.} When 
\begin{equation*}
   z \in \C\backslash \Sigma(p),
\end{equation*}
then by condition \eqref{C1:eq5.2} we have that $(p_0(\rho)-z)\geq m(\rho)/C$ 
for some sufficiently large $C>0$ and so we know that the inverse 
$(P_h-z)^{-1}$ is a pseudo-differential operator with principal 
symbol $(p_0-z)^{-1} \in S(1/m)\subset S(1)$. Hence, $(P_h-z)^{-1}$ 
maps $L^2\to L^2$ and 
\begin{equation}
   \|(P_h-z)^{-1}\| =\mO(1)
\end{equation} 
uniformly in $h>0$. Hence, from the semiclassical point of view we may 
consider $\C\backslash \Sigma$ as a \emph{zone of spectral stability}. 
\subsection{At the boundary of the semiclassical pseudospectrum}
\label{C2Susec:ResEstBd}
At the boundary of the semiclassical pseudospectrum we find a transition 
between the zone of strong spectral instability and stability. Indeed at 
the boundary we find an improvement on the resolvent bounds, assuming some 
additional non-degeneracy:
\par
Splitting a symbol $p\in C^\infty_b(T^*\R^d)$ into real and imaginary part, 
$p=p_1+ip_2$, we consider the iterated Poisson bracket 
\begin{equation*}
	p_I := \{p_{i_1},\{p_{i_2},\{\dots,\{p_{i_{k-1}},p_{i_k}\}\}\dots \}\}
\end{equation*}
where $I\in\{1,2\}^k$, and $|I|=k$ is called the \emph{order} of the Poisson 
bracket. The \emph{order} of $p$ at $\rho\in T^*\R^d$ is given by 
\begin{equation*}
	k(\rho) := \max\{ j\in\N; p_I(\rho)=0,~1<|I|\leq j\}.
\end{equation*}
The \emph{order} of $z_0\in\Sigma\backslash \Sigma_\infty$ is the maximum 
of $k(\rho)$ for $\rho\in p^{-1}(z_0)$. 
\begin{thm}{\cite{NSjZw04,Sj10a}}\label{thm:DZSboundary}
	Assume that $C^\infty_b(T^*\R^d)\ni p\sim p_0+hp_1+\dots$. Let 
	$P_h = p^w(x,hD_x)$, and let 
	$z_0\in \partial \Sigma(p_0) \backslash \Sigma_\infty(p_0)$. Suppose that 
	$dp_0\neq 0$ at every point in $p^{-1}_0(z_0)$, that $z_0$ is 
	of finite order $k\geq 1$ for $p$. Then, $k$ is even and 
	for $h>0$ small enough 
	\begin{equation*}
		\| (P_h-z)^{-1}\| \leq Ch^{-\frac{k}{k+1}}.
	\end{equation*}
	In particular, there exists a $c_0>0$, such that for $h>0$ small 
	enough 
	\begin{equation*}
		\{z\in\C; |z-z_0| \leq c_0h^{\frac{k}{k+1}} \} \cap \spec(P_h) =
		\emptyset. 
	\end{equation*}
\end{thm}
This result was proven in dimension $1$ by Zworski \cite{Zw01a}, and in certain 
cases by Boulton \cite{Bo02}. Further refinements have been obtained in 
\cite{Sj10a}. Similar to the discussion after Theorem \ref{thm:DZS}, we can extend 
Theorem \ref{thm:DZSboundary} to unbounded symbols $p\in S(T^*\R^{d},m)$ and their 
corresponding quantizations. 
\begin{ex}\label{ex:DaviesOperator2}
	Recall the non-selfadjoint Harmonic oscillator $P_h=(hD_x)^2+ix^2$ 
	from Example \ref{ex:DaviesOperator}. Here $\partial\Sigma = \R_+ \cup i\R_+$, 
	so we see by \eqref{C1eq:DaviesOp1.1} that for $0\neq z_0 \in \Sigma$ 
	\begin{equation*}
		\frac{1}{2i}\{\overline{p},p\}(\rho) 
		= \{\Rea p,\Ima p\}(\rho) = 0, \quad \rho \in p^{-1}(z_0). 
	\end{equation*}
	However, there 
	\begin{equation*}
		\text{either } \{\Rea p,\{\Rea p,\Ima p\}\} (\rho) = 4\xi^2\neq 0, \quad 
		\text{or }\{\Ima p,\{\Rea p,\Ima p\}\} (\rho) = -4x^2\neq 0,
	\end{equation*}
	so $z_0$ is of order $2$ for $p=\xi^2+ix^2$, and Theorem \ref{thm:DZSboundary} 
	tells us that 
	\begin{equation*}
		\| (P_h-z_0)^{-1}\| \leq Ch^{-\frac{2}{3}}.
	\end{equation*}
	We see that for a the $\varepsilon$-pseudospectrum of $P_h$ to 
	reach the boundary of $\Sigma$, we require $\varepsilon > h^{2/3}/C$.
\end{ex}
\subsection{Pseudospectra and random matrices}
In this section we present a short discussion of pseudospectra for large 
$N\times N$ random matrices. We can interpret the $1/N$, for $N\gg1$, as 
an analogue to the semiclassical parameter. Recalling the example of the 
non-selfadjoint 
harmonic oscillator, see Example \ref{ex:DaviesOperator}, we see that 
pseudospectra can be very large in general. However, in a generic setting 
they are typically much smaller. 
\par 
Let $M\in \C^{N\times N}$ be a complex $N\times N$ matrix and let 
$s_1(M)\geq \dots \geq s_N(M)\geq 0$ denote its singular values, 
i.e. the eigenvalues of $\sqrt{M^*M}$ ordered in a decreasing manner 
and counting multiplicities. Note that if $M-z$ is bijective for some 
$z\in\C$, then 
\begin{equation*}
	\|(M-z)^{-1}\| = s_N(M-z)^{-1}. 
\end{equation*}
In view of \eqref{C1:eq1}, the $\varepsilon$-pseudospectrum of $M$ 
is then characterized by the condition that $z\in \spec_\varepsilon(M)$ 
\begin{equation*}
	z\in \spec_\varepsilon(M) 
	\quad 
	\Longleftrightarrow 
	\quad 
	s_N(M-z)<\varepsilon. 
\end{equation*}
A classical result due to Sankar, Spielmann and Teng 
\cite[Lemma 3.2]{SaTeSp06} (stated there for real 
Gaussian random matrices) tells us that with high 
probability the smallest singular value of a deformed 
random matrix is not too small. 
\begin{thm}[{\cite{SaTeSp06} }]\label{thm:SSV}
There exists a constant $C>0$ such that the following holds. 
Let $N\geq 2$, let $X_0$ be an arbitrary complex $N\times N$ matrix, 
and let $Q$ be an $N\times N$ complex 
Gaussian random matrix whose entries are 
all independent copies of a 
complex Gaussian random variable $q\sim \mathcal{N}_{\C}(0,1)$. 
Then, for any $\delta >0$
\begin{equation*}
	\prob \left(s_N(X_0 + \delta Q) < \delta t \right) \leq C
	N t^2.
\end{equation*}
\end{thm}
\begin{proof}
For real matrices the proof can be found in \cite[Lemma 3.2]{SaTeSp06}, 
see also \cite[Theorem 2.2]{TaVu10}. For complex matrices a proof is 
presented for instance in \cite[Appendix A]{Vo20}.
\end{proof}
Theorem \ref{thm:SSV} tells us that any fixed $z\in \C$ is not in the 
$\varepsilon$-pseudospectrum of $X+\delta Q$ with probability 
$\geq 1 - CN\varepsilon^2\delta^{-2}$. We can interpret this 
result as saying that the pseudospectrum of random matrices 
is typically \emph{not too large}. 
%
Theorem \ref{thm:SSV} has enjoyed many extensions. For instance Rudelson and Vershynin 
\cite{RV08} consider the case random matrices with iid sub-Gaussian entries. 
Tao and Vu \cite{TaVu08} consider iid entries of non-zero variance. 
Cook \cite{Co18} consider the case of random matrices whose of entries have an 
inhomogeneous variance profile under appropriate assumptions. We end this section 
by noting the following, quantitative result due to Tao and Vu. 
\begin{thm}[{\cite{TaVu10}}]\label{thm:TavoVu}
Let $q$ be a random variable with mean zero and bounded second moment, 
and let $\gamma\geq 1/2$, $A\geq 0$ be constants. Then, there exists 
a constant $C>0$, depending on $q,\gamma,A$ such that the following 
holds. Let $Q$ be the random matrix of size $N$ whose entries are independent 
and identically distributed copies of $q$, let $X_0$ be a deterministic 
matrix satisfying $\|X_0\| \leq N^\gamma$. Then, 
\begin{equation}\label{gc2.4.2}
\textbf{P}\left( s_n(X_0 + Q) \leq n^{-\gamma (2A+2) +1/2} \right) 
\leq  C \left( n^{-A + o(1)} + \prob( \| Q\| \geq n^{\gamma})\right).
\end{equation}
\end{thm}
\begin{ex}
Consider the case where $q$ is a random variable satisfying the 
moment conditions 
\begin{equation}\label{gc2.4.3}
	\erw[ q ] =0, \quad  \erw[|q|^2] =1, \quad \erw[ |q|^4] < +\infty.
\end{equation}
Form \cite{La05} we know that \eqref{gc2.4.3} implies that 
$\erw [ \| Q\| ] \leq C N^{1/2}$, which,    
using Markov's inequality, yields that for any $\varepsilon>0$
\begin{equation}\label{gc2.4.5}
\prob\left[ 
\Vert Q\Vert \geq C N^{1/2+\varepsilon}
\right] 
\leq C^{-1}N^{-1/2- \varepsilon}\erw [ \| Q\| ] 
\leq N^{-\varepsilon}.
\end{equation}
In this case \eqref{gc2.4.2} becomes 
\begin{equation}\label{gc2.4.2b}
	\textbf{P}\left( s_n(X_0 + Q) \leq n^{-(\varepsilon+1/2)(2A+2) +1/2} \right) 
	\leq  C \left( n^{-A + o(1)} + N^{-\varepsilon}\right).
\end{equation}
\end{ex}
\section{Eigenvalue asymptotics for non-selfadjoint (random) operators}
Consider the operator $P_h=p^w(x,hD_x)$ as in \eqref{C1:eq5}, \eqref{C1:eq8}, 
seen as an unbounded operator $L^2(\R^d)\to L^2(\R^d)$. We equip $P_h$ with the 
domain $H(m):=(P_h-z_0)^{-1}L^2(\R^d)$. Note that $(P_h-z_0)^{-1}$ exists for $h>0$ 
small enough by the elipticity condition \eqref{C1:eq5.2}. We will denote 
by $\|u\|_m:= \| (P_h-z_0) u \|$ the associated norm on $H(m)$. Although this norm 
depends on the choice of the symbol $p_0-z_0$, it is equivalent to the norm defined 
by any operator with elliptic principal symbol $q\in S(m)$, so that the space $H(m)$ 
only depends on the order function $m$. Since $H(m)$ contains the 
Schwartz functions $\cS(\R^d)$ it is dense in $L^2(\R^d)$. 
\par
Let us check that $P_h$ equipped with domain $H(m)$ is closed. Let $(P_h-z_0)u_j \to v$ and 
$u_j \to u$ in $L^2$. Since $(P_h-z_0):H(m)\to L^2$ is bijective, it follows that 
$u_j \to (P_h-z_0)^{-1}v$ in $H(m)$ and also in $L^2$. So $u = (P_h-z_0)^{-1}v$. 
Summing up, $P_h$ equipped with domain $H(m)$ is a densely defined closed linear 
operator. 
\\
\par 
Recall \eqref{C1:eq5.1}, and let 
\begin{equation}\label{C2eq:WA1}
	\Omega\Subset \C\backslash \Sigma_\infty
\end{equation}
be not entirely contained in $\Sigma$. Using the ellipticity assumption 
\eqref{C1:eq5.2} it was proven in \cite[Section 3]{HaSj08} that 
\begin{itemize}
	\item $\spec(P_h)\cap \Omega$ is discrete for $h>0$ small enough, 
	\item For all $\varepsilon >0$ there exists an $h(\varepsilon)>0$ such that 
		\begin{equation*}
			\spec(P_h)\cap \Omega \subset \Sigma+ D(0,\varepsilon), \quad 0< h \leq h(\varepsilon),
		\end{equation*}
		where $D(0,\varepsilon)$ denotes the disc in $\C$ of radius $\varepsilon$ and centered at 
		$0$. 
\end{itemize}
\subsection{The selfadjoint setting}
If $P_h$ above is selfadjoint, which implies in particular that $p$ is real-valued, 
we have the classical Weyl asymptotics. We follow here \cite{DiSj99} 
for a brief review. 
\begin{thm}\label{thm:WeylAsymptotics1}
	Let $\Omega$ be as in \eqref{C2eq:WA1}. 
	For every $h$-independent interval $I\subset \Omega\cap \R$ with 
	$\Vol_{\R^{2d}}(\partial I) =0$, 
	\begin{equation}
		\#(\spec(P_h)\cap I) = \frac{1}{(2\pi h)^d}\left(\int_{p_0^{-1}(I)} dx d\xi +o(1)\right), 
		\quad h\to 0.
	\end{equation}
\end{thm}
This result is, in increasing generality, due to Chazarin \cite{Ch80}, Helffer and Robert \cite{HeRo81,HeRo83}, 
Petkov and Robert \cite{PeRo85} and Ivrii \cite{Iv98}. See also \cite{DiSj99} for an overview. We highlight 
two special cases: when $I=[a,b]$, $a<b$, and $a,b$ are not critical points of $p_0$, then 
the error term becomes $\mO(h)$, see Chazarin \cite{Ch80}, Helffer-Robert \cite{HeRo81} and Ivrii \cite{Iv98}. When 
additionally the unions of periodic $H_{p_0}$ trajectories\footnote{$H_{p_0}$ denotes the Hamilton vector field induced 
by $p_0$.} in the energy shell $p_0^{-1}(a)$ and 
$p_0^{-1}(b)$ are of Liouville measure $0$, then the error term is of the form 
\begin{equation}
	 h \left(
		\int_{p_0=a} p_1(\rho ) L_a(d\rho) - 
		\int_{p_0=b} p_1(\rho ) L_b(d\rho) 
	 \right) + o(h), 
\end{equation}
where $L_\lambda$ denotes the Liouville measure on $p^{-1}_0(\lambda)$. See Petkov and Robert \cite{PeRo85} 
and Ivrii \cite{Iv98}, as well as \cite{DiSj99}, for details. Let us also highlight that similar results to 
Theorem \ref{thm:WeylAsymptotics1} also hold on compact smooth manifolds, see for instance 
\cite[Chapter 12]{GrSj94} and the references therein. 
\par
The corresponding results in the setting of selfadjoint partial differential 
operators in the high energy limit go back to the seminal work of Weyl 
\cite{We12} and have a long and very rich history. These are however beyond 
the scope of this review. 
\begin{ex}\label{ex:HarmonicOscillator}
	The guiding example to keep in mind is the selfadjoint Harmonic oscillator 
	\begin{equation*}
	   P_h = (hD_x)^2 + x^2:~L^2(\R)\to L^2(\R)
	\end{equation*}
	seen as an unbounded operator. The principal symbol 
	of $P_h$ is given by $p(x,\xi)=\xi^2+x^2\in S(T^*\R,m)$, with weight 
	function $m(x,\xi) = 1+ \xi^2 + x^2$. We equip $P_h$ with the domain 
	$H(m):=(P_h+1)^{-1}L^2(\R)$, where the operator on the right is the 
	pseudo-differential inverse of $P_h+1$. This choice of domain makes $P_h$ 
	a closed densely defined operator. It is well-known (see for instance 
	\cite[Theorem 6.2]{Zw12}) that the spectrum of $P_h$ is given by 
	\begin{equation*}
	   \spec(P_h) = \{(2n+1)h;n\in\N\}.
	\end{equation*}
	Counting the points $(2n+1)h$ contained in an interval $[a,b]$, $0\leq a<b<\infty$, 
	gives 
	\begin{equation*}
		\#(\spec(P_h)\cap [a,b]) = \frac{b-a}{2h} + \mO(1).
	 \end{equation*}
	Since $\mathrm{Vol}_{\R^2}(\{a\leq \xi^2+x^2\leq b\}) = \pi (b-a)$, we confirm 
	Theorem \ref{thm:WeylAsymptotics1} for the Harmonic oscillator. 
\end{ex}
\subsection{The non-selfadjoint setting}\label{C2sec:NSAsetting}
The natural counterpart of Theorem \ref{thm:WeylAsymptotics1} for non-selfadjoint 
operators would be eigenvalue asymptotics in a complex domain $\Omega\Subset\C$ 
as in \eqref{C2eq:WA1}. Recall the non-selfadjoint Harmonic oscillator $P_h$ from 
Example \ref{ex:DaviesOperator} with principal symbol $p(x,\xi)=\xi^2+ix^2$. 
In this case $\Sigma=\{z\in\C; \Rea z,\Ima z\geq 0\}$ and $\Sigma_\infty=\emptyset$. Any  
$\emptyset \neq \Omega\Subset \Sigma$ away from the line $\e^{i\pi/4}\R_+$ gives in view of 
\eqref{C1eq:DaviesOp1} that 
\begin{equation*}
	\#(\spec(P_h)\cap \Omega) = 0.
\end{equation*}
On the other hand 
\begin{equation*}
	\frac{1}{2\pi h}\int_{p^{-1}(\Omega)} dx d\xi >0.
\end{equation*}
This example suggests that a direct generalization of Theorem 
\ref{thm:WeylAsymptotics1} to non-selfadjoint operators with complex 
valued principal symbol cannot hold. 
\\
\par
%
Let us comment on two settings where a form of Weyl asymptotics is 
known to hold: Upon assuming analyticity, one may recover a sort of Weyl asymptotics. 
More precisely, as shown in the works of Melin and Sj\"ostrand \cite{MeSj03}, Sj\"ostrand \cite{Sj03}, 
Hitrik and Sj\"ostrand \cite{HS04,HS05,HS08}, Hitrik, Sj\"ostrand and V{\~u} Ng{\d o}c \cite{HSN07} and 
Rouby \cite{Ro18}, the discrete spectrum of certain analytic non-selfadjoint 
pseudo-differential operators is confined to curves in $\Sigma$. Moreover, one may  
recover eigenvalue asymptotics via Bohr-Sommerfeld quantization conditions. 
\\
\par
The second setting is when the non-selfadjointness of the operator $P_h$ 
comes not from the principal symbol $p_0$ (assumed to be real-valued) but from the 
subprincipal symbol $p_1$. For instance, when studying the damped wave equation on a compact 
Riemannian manifold $X$ one is led to study the eigenvalues of the 
corresponding stationary operator 
\begin{equation*}
	P_h(z) = -h^2\Delta +2ih\sqrt{a(x)}\sqrt{z}, \quad a\in C^\infty(X;\R).
\end{equation*}
Here, $\Delta$ denotes the Laplace-Beltrami operator on $X$ and 
we call $z\in\C$ an eigenvalue of $P_h(z)$ if there exists a corresponding 
$L^2$ function $u$ contained in the kernel of $P_h(z)-z$. Actually, such a 
$u$ is smooth by elliptic regularity. Using Fredholm theory one can show 
that these eigenvalues form a discrete set in $\C$. 
\par
The principal part of $P_h=P_h(z)$ is given by $-h^2\Delta$, and 
thus is self-adjoint, with principal symbol is $p_0(x,\xi)=|\xi|^2_x$ 
(the norm here is with respect to the Riemannian metric on $X$). However, 
the subprincipal part is complex valued and non-selfadjoint. 
\par
Lebeau \cite{Le96} established that there exist $a_\pm\in \R$ such that 
for every $\varepsilon>0$ there are only finitely many eigenvalues such 
that 
\begin{equation*}
	\frac{\Ima z}{h} \notin [a_--\varepsilon,a_++\varepsilon].
\end{equation*}
\begin{rem}
	In fact Lebeau provided precise expressions for $a_\pm$ in 
	terms of the infimum and the supremum over the co-sphere bundle $S^*X$ 
	of the long time average of the damping function $a$ evolved via the 
	geodesic flow. Further refinements have been obtained by Sj\"ostrand 
	\cite{Sj00}, and when $X$ is negatively curved by Anantharaman \cite{An10} 
	and Jin \cite{Ji20}.
\end{rem}
Additionally Markus and Matsaev \cite{MaMa82} and Sj\"ostrand \cite{Sj00} 
proved the following analogue of the Weyl law. For $0<E_1< E_2<\infty$ and 
for $C>0$ sufficiently large
\begin{equation}\label{C2eq:WeylAsymptoticsDW}
	\#\big(\spec(P_h)\cap ([E_1,E_2]+ i[-Ch,Ch]) \big) 
	=\frac{1}{(2\pi h)^d}
	\left(\iint_{p_0^{-1}([E_1,E_2])}dxd\xi + \mO(h)\right).
\end{equation}
Finer results have been obtained by Anantharaman \cite{An10} 
and Jin \cite{Ji20} when $X$ is negatively curved. 

\subsection{Probabilistic Weyl asymptotics}
In a series of works by Hager \cite{Ha06,Ha06b,HaSj08} and 
Sj\"ostrand \cite{Sj09,Sj10}, the authors proved a Weyl law, 
with overwhelming probability, for the eigenvalues in a compact 
set $\Omega\Subset \C$ as in \eqref{C2eq:WA1} for randomly 
perturbed operators 
\begin{equation}\label{C2eq:PertOp1}
	P^{\delta} = P_h+\delta Q_{\omega},\quad 0<\delta=\delta(h)\ll1,
\end{equation}
where $P_h$ is as in Section \ref{sec:SpecInstaSemPseudo} and 
the random perturbation $Q_\omega$ is one of the following two types. 
\\[2ex]
\textbf{Random Matrix.} 
Let $N(h)\to \infty$ sufficiently fast as $h\to 0$. Let $q_{j,k}$, 
$0\leq j,k <N(h)$ be independent copies of a complex Gaussian random 
variable $\alpha\sim \mathcal{N}_\C(0,1)$. We consider the random matrix
\begin{equation}\label{C2eq:PertOp2}
	Q_{\omega}=\sum_{0\leq j,k<N(h)} q_{j,k}\, e_j \otimes e_k^*,
\end{equation}
where $\{e_j\}_{j\in\N}\subset L^2(\R^d)$ is an orthonormal basis and 
$e_j \otimes e_k^*u=(u|e_k)e_j$ for $u\in L^2(\R)$. The condition on $N(h)$ 
is determined by the requirement that the microsupport of the vectors in the 
orthonormal system $\{e_j\}_{j< N(h)}$ ``covers'' the compact set 
$p_0^{-1}(\Omega)\subset T^*\R^d$, where $p_0$ is the principal symbol of $P_h$. 
For instance, we could take the first $N(h)$ eigenfunctions (ordered according to 
increasing eigenvalues) of the Harmonic oscillator $P_h = -h^2\Delta + x^2$ on $\R^d$. 
The number $N(h)$ is then determined by the condition that the semiclassical wavefront sets 
of $e_{j}$, $j\geq N(h)$, are disjoint from $p_0^{-1}(\Omega)$. 
Alternatively, as in \cite{HaSj08}, one may 
take $N(h)=\infty$, however then one needs to conjugate $Q_\omega$ by suitable 
elliptic Hilbert-Schmidt operators. We refer to \cite{HaSj08} for more details. 
\\[2ex]
\textbf{Random Potential.} We take $N(h)$ and an orthonormal family 
$(e_k)_{k\in\N}$ as above. Let $v$ be real or complex random 
vector in $\R^{N(h)}$ or $\C^{N(h)}$, respectively, with joint probability law
\begin{equation}\label{C2eq:PertOp3.0}
	v_*(d \prob) =  
	Z_h^{-1}\, \mathbf{1}_{B(0,R) }(v) \,\e^{\phi(v)} L(dv),
\end{equation} 
where $Z_h>0$ is a normalization constant, $B(0,R)$ is either the real 
ball $\Subset \R^{N(h)}$ or the complex ball $\Subset \C^{N(h)}$ of radius 
$R =R(h)\gg 1$, and centered at $0$, $L(dv)$ denotes the Lebesgue
measure on either $\R^{N(h)}$ or $\C^{N(h)}$ and $\phi \in C^1$ with 
\begin{equation}\label{C2eq:PertOp3.1}
	\| \nabla_v \phi \| = \mO(h^{- \kappa_4}) 
\end{equation}
uniformly, for an arbitrary but fixed $\kappa_4\geq 0$. In \cite{Ha06} 
the case of non-compactly supported probability law was considered. More 
precisely, the entries of the random vector $v$ were supposed to be 
independent and identically distributed (iid) complex Gaussian random variables 
$\sim \mathcal{N}_\C(0,1)$. In \cite{Sj09,Sj10}, the law \eqref{C2eq:PertOp3.0} 
was considered. For the sake 
of simplicity we will not detail here the precise conditions on the 
$e_k$, $R(h)$, and $N(h)$, in this case but refer the reader to \cite{Sj09,Sj10}. 
However, one example of a random vector $v$ with law \eqref{C2eq:PertOp3.1}
is a truncated complex or real Gaussian random variables with expectation 
$0$, and uniformly bounded covariances. In fact, the methods in \cite{Sj09,Sj10} 
can be extended to non-compactly supported probability distributions, provided 
sufficient decay conditions at infinity are assumed. For instance iid complex 
Gaussian random variables, as in the one dimensional case \cite{Ha06}, are 
permissable. Finally, we remark 
that the methods in \cite{Sj09,Sj10} can probably also be modified to allow 
for the case of more general independent and identically distributed random 
variables. 
\par
We define the random function  
\begin{equation}\label{C2eq:PertOp3}
	V_{\omega}=\sum_{0\leq j<N(h)} v_{j}\, e_j.
\end{equation}
We call this perturbation a ``random potential'', even though $V_\omega$ is 
complex valued. When we consider this type of perturbation, we will make the 
additional symmetry assumption:
\begin{equation}\label{C2eq:PertOp4}
	p(x,\xi;h)=p(x,-\xi;h).
\end{equation}
%
Let $\Omega\Subset \C$ be an open simply connected set as in \eqref{C2eq:WA1}. 
For $z\in \Omega$ and $0\leq  t \ll1 $ we set 
\begin{equation}\label{C2eq:PertOp5}
		V_z(t) = \mathrm{Vol}\{\rho \in T^*\R^d; |p_0(\rho) -z |^2 \leq t \}.
\end{equation}
Let $\Gamma\Subset \Omega$ be open with $\mathcal{C}^2$ boundary and 
make the following non-flatness assumption 
\begin{equation}\label{C2eq:PertOp6}
		\exists \kappa \in ]0,1], \text{ such that } V_z(t) = \mO(t^{\kappa}), 
		\text{ uniformly for } z \in \text{neigh}(\partial \Gamma), ~ 0 \leq t \ll 1.
\end{equation}
The above mentioned works shown the following result.
\begin{thm}[Probabilistic Weyl's law]\label{thm:PWL}
Let $\Omega$ be as in \eqref{C2eq:WA1}. Let $\Gamma\Subset\Omega$ 
be open with $\mathcal{C}^2$ boundary. Let $P^{\delta}_h$ be a randomly 
perturbed operators as in \eqref{C2eq:PertOp1} with 
$\e^{-1/Ch} \ll \delta\leq h^\theta$ with $\theta>0$ sufficiently large. 
Then, in the limit $h\to 0$,
\begin{equation}
	\#\big(\spec(P^{\delta}_h)\cap \Gamma\big) = \frac{1}{(2\pi h)^d}
		\left( \iint_{p_0^{-1}(\Gamma)} dxd\xi + o(1) \right)\quad \text{with probability $\geq 1 - C h^{\eta}$},
\end{equation}
for some fixed $\eta>0$.
\end{thm}
The works \cite{Ha06,Ha06b,HaSj08,Sj09,Sj10} also provide an explicit control 
over $\theta$, the error term in Weyl's law, and the error term in the 
probability estimate. We illustrate Theorem \ref{thm:PWL} with a numerical 
simulation in Figure \ref{PWL.fig1} below. 
\begin{figure}[ht]
\centering
 \begin{minipage}[b]{0.45\linewidth}
  \centering
  \includegraphics[width=\textwidth]{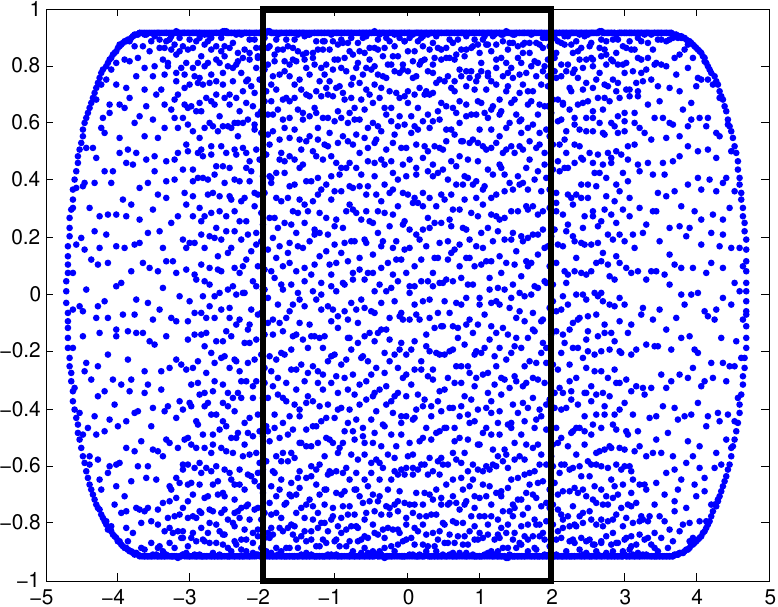}
 \end{minipage}
 \hspace{0.1cm}
 \begin{minipage}[b]{0.45\linewidth}
  \centering 
  \includegraphics[width=\textwidth]{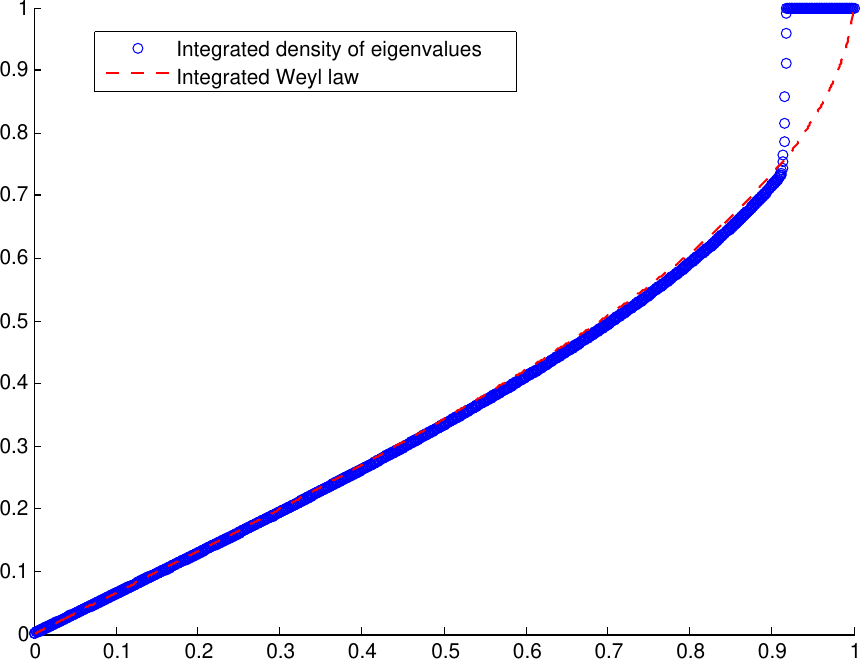}
 \end{minipage}
 \caption{The left hand side depicts the spectrum of a discretization of 
	$P_h=hD_x + \exp(-ix)$, $x\in S^1$, (approximated by a $N\times N$-matrix, 
	$N=3999$) perturbed with a random Gaussian matrix $\delta Q_\omega$ with 
	$h\asymp N^{-1}$ and $\delta\asymp N^{-4}$. The spectrum of the perturbed 
	operator differs strongly from the spectrum of the unperturbed operator 
	which is given by$\{hk;k\in \Z\}$. The right hand side shows the integrated 
	experimental density of eigenvalues (as a function 
	of the imaginary part) in the black box, averaged over $400$ realizations of 
	random Gaussian matrices. The red line shows the integrated Weyl 
	law which coincides with the experimental density in the interior 
	of the classical spectrum $\Sigma$. The two densities differ close to 
	boundary\footref{fn3} $\partial\Sigma$. This figure stems form 
	\cite{Vo17a}.}\label{PWL.fig1}
\end{figure}
\par
\footnotetext[3]{For finer results concerning the brake-down of the Weyl law 
on a mesoscopic scale close to the pseudospectral boundary we refer the reader 
to \cite{Vo17a}.\label{fn3}}
Theorem \ref{thm:PWL} is remarkable because 
such Weyl laws are typically a feature of selfadjoint operator, whereas in 
the non-selfadjoint case they generally fail. Indeed, as laid out in Section 
\ref{C2sec:NSAsetting}, the discrete spectrum of the (unperturbed) non-selfadjoint 
operator $P_h$ is usually localized to curves in the pseudospectrum $\Sigma$, see 
\cite{MeSj03,HS04,HS05,HS08,HSN07,Ro18}. In contrast, Theorem \ref{thm:PWL} 
shows that a ``generic'' perturbation of size $\mathcal{O}(h^{\infty})$ is sufficient for the 
spectrum to ``fill out'' $\Sigma$. 
\par 
To illustrate this phenomenon recall the non-selfadjoint harmonic oscillator 
$P_h  = -h^2\partial_x^2 + ix^2$ on $\R$ from Example \ref{ex:DaviesOperator}.  
Its spectrum is given by $\{\e^{i\pi/4}(2n+1)h; n\in\N\}$ \cite{Da99b} on the 
line $\e^{i\pi/4}\R_+\subset \C$. The Theorem \ref{thm:PWL} shows that a 
``generic'' perturbation of arbitrarily small size is sufficient to produce 
spectrum roughly equidistributed in any fixed compact set in its classical 
spectrum $\Sigma$, which is in this case the upper right quadrant of $\C$. 
\\
\par
As observed in \cite{ChZw10} for real analytic $p$ condition 
\eqref{C2eq:PertOp6} always holds for some $\kappa >0$. Similarly, when $p$ is 
real analytic and such that $\Sigma\subset\C$ has non-empty interior, 
then 
\begin{equation}\label{ke1}
	\forall z\in\partial\Omega: ~
	dp \!\upharpoonright_{p^{-1}(z)} \neq 0 \quad \Longrightarrow \quad 
	\eqref{C2eq:PertOp6} \text{ holds with } \kappa > 1/2.
\end{equation}
For smooth $p$ we have that when for every $z\in\partial\Omega$ 
\begin{equation}\label{ke2}
\begin{split}
	&dp, d\overline{p} \text{ are linearly independent at every point of } p^{-1}(z),  \\
	&\text{then }\eqref{C2eq:PertOp6} \text{ holds with } \kappa = 1.
\end{split}
\end{equation}
Observe that $dp$ and $d\overline{p}$ are linearly independent at $\rho$ when 
 $\{p,\overline{p}\}(\rho)\neq 0$, 
 where $\{a,b\} = \partial_\xi a \cdot \partial_x b -  \partial_x a \cdot \partial_\xi b$ 
denotes the Poisson bracket. Moreover, in dimension $d=1$ 
the condition $\{p,\overline{p}\}\neq 0$ on $p^{-1}(z)$ is equivalent to 
$dp$, $d\overline{p}$ being linearly independent at every point of $p^{-1}(z)$. 
However, in dimension $d>1$ this cannot in hold general as the integral of 
$\{p,\overline{p}\}$ with respect to the Liouville measure on $p^{-1}(z)$ vanishes 
on every compact connected component of $p^{-1}(z)$, see \cite[Lemma 8.1]{MeSj02}. 
Furthermore, condition \eqref{ke2} cannot hold when $z\in \partial\Sigma$. 
However, some iterated Poisson bracket may not be zero there. For example, 
it was observed in \cite[Example 12.1]{HaSj08} that if
\begin{equation}\label{ke3}
	\forall \rho \in p^{-1}(\partial \Omega): ~\{p,\overline{p}\}(\rho) \neq 0 \text{ or }
	\{p,\{p,\overline{p}\}\} (\rho) \neq 0, 
	\text{ then }
	\eqref{C2eq:PertOp6} \text{ holds with } \kappa = \frac{3}{4}.
\end{equation}
\textbf{Related results.}
Theorem \ref{thm:PWL} has also been extended to case of elliptic semiclassical 
differential operators on compact manifolds by Sj\"ostrand \cite{Sj10}, to Toeplitz 
quantization of the torus by Christiansen and Zworski \cite{ChZw10} and the 
author \cite{Vo20}, and to general Berezin-Toeplitz 
quantizations on compact K\"ahler manifolds by Oltman \cite{Ol23} in the 
case of complex Gaussian noise. A further extension of Theorem \ref{thm:PWL} 
has been achieved by Becker, Oltman and the author in \cite{BOV24}. There 
we prove a probabilistic Weyl law for the non-selfadjoint off-diagonal operators  
of the Bistritzer-MacDonald Hamiltonian \cite{BM11} for twisted bilayer 
graphene, see also \cite{CGG,Wa22}, subject to random tunneling potentials. 
This probabilistic Weyl has an interesting physical consequence as it shows 
the instability of the so-called \emph{magic angels} for this model of twisted 
bilayer graphene.  
\par 
Similar results have been obtained in random matrix theory. 
The case of Toeplitz matrices given by symbols on $\T^2$ of the form 
$\sum_{n\in\Z} a_n\e^{in\xi}$, $(x,\xi)\in \T^2$, 
was studied in a series of recent works by {\'S}niady \cite{Sn02}, 
Davies and Hager \cite{DaHa09}, Guionnet, Wood and Zeitouni \cite{GuMaZe14}, 
Basak, Paquette and Zeitouni \cite{BPZ18, BPZ19}, Sj\"ostrand and the author 
of this text \cite{SjVo15b,SjVo19a,SjVo19b}. Such symbols amount to the case of symbols which are 
constant in the $x$ variable. In these works the non-selfadjointness of the problem does however not 
come from the symbol itself but from boundary conditions destroying the periodicity of the symbol 
in $x$ by allowing for a discontinuity. Nevertheless, these works show that by adding some small random 
matrix the limit of the empirical eigenvalues counting measure $\mu_N$ of the perturbed operator converges 
in probability (or even almost surely in some cases) to $p_*(d\rho)$. 
\providecommand{\bysame}{\leavevmode\hbox to3em{\hrulefill}\thinspace}
\providecommand{\MR}{\relax\ifhmode\unskip\space\fi MR }
\providecommand{\MRhref}[2]{%
  \href{http://www.ams.org/mathscinet-getitem?mr=#1}{#2}
}
\providecommand{\href}[2]{#2}

\end{document}